\documentclass[aps,11pt,twoside, nofootinbib, superscriptaddress]{revtex4}
\usepackage{amsthm}
\usepackage{amsmath,latexsym,amssymb,verbatim,enumerate,graphicx}
\usepackage{color}
\usepackage[utf8]{inputenc}
\usepackage[T1]{fontenc}
\usepackage[polish]{babel}

\newtheorem{definition}{Definition} %[section]

\newtheorem{lemma}[definition]{Lemma}

\newtheorem{thm}[definition]{Theorem}

\newtheorem{remark}[definition]{Remark}

\newtheorem*{rep@theorem}{\rep@title}
\newcommand{\newreptheorem}[2]{%
\newenvironment{rep#1}[1]{%
 \def\rep@title{#2 \ref{##1} (restatement)}%
 \begin{rep@theorem}}%
 {\end{rep@theorem}}}
\makeatother

\newreptheorem{thm}{Theorem}
\newreptheorem{lem}{Lemma}

\def\ba#1\ea{\begin{align}#1\end{align}}
\def\ban#1\ean{\begin{align*}#1\end{align*}}

\newcommand{\be}{\begin{equation}}
\newcommand{\ee}{\end{equation}}

\begin{document}
\title{Limit theorems for some Markov operators}
\author{Sander Hille}
\affiliation{Mathematical Institute, Leiden University, P.O. Box 9512, 2300 RA Leiden, The Netherlands}
%\email{shille@math.leidenuniv.nl}
\author{Katarzyna Horbacz}
\affiliation{Institute of Mathematics, University of Silesia, Bankowa 14, 40-007 Katowice, Poland}
%\email{horbacz@math.us.edu.pl}
\author{Tomasz Szarek}
\address{Institute of Mathematics, University of Gda\'nsk, Wita Stwosza 57, 80-952 Gda\'nsk, Poland}
%\email{szarek@intertele.pl}
\author{Hanna Wojew\'odka}
\affiliation{Institute of Mathematics, University of Gda\'nsk, Wita Stwosza 57, 80-952 Gda\'nsk, Poland}
\affiliation{National Quantum Information Center of Gda\'{n}sk, 81-824 Sopot, Poland}
\affiliation{Institute of Theoretical Physics and Astrophysics, University of Gda\'{n}sk, 80-952 Gda\'{n}sk, Poland}
%\email{hwojewod@mat.ug.edu.pl}

\begin{abstract} 
The exponential rate of convergence and the Central Limit Theorem for some Markov operators are established. The operators correspond to iterated function systems which, for example, may be used to generalize the cell cycle model given by Lasota and Mackey \cite{lasotam}.
\end{abstract}

%\subjclass[2000]{60J25 (primary), 37A25 (secondary)}
%\keywords{ Exponential rate of convergence, central limit theorem, Markov operators, invariant measures}

%\date{\today}
\thanks{The work of Tomasz Szarek has been partly supported by the National Science Centre of Poland, grant number DEC-2012/07/B/ST1/03320.} 
\maketitle

\section{Introduction}

We are concerned with Markov operators corresponding to iterated function systems. The main goals of the paper are to prove exponential rate of convergence and establish the Central Limit Theorem (CLT). It should be indicated that the first result implies the second.  The operators under consideration are more general than those used by Lasota and Mackey in \cite{lasotam}. The authors studied therein some cell cycle model, in which the rate of convergence is already evaluated by Wojew\'odka \cite{hw}. Hille at el. \cite{hille} proposed the generalization of the model and assured the existence of a unique invariant distribution in it. We have managed to evaluate the rate of convergence, which provides asymptotic stability at once, as well as allows us to show the CLT. The results bring some information important from biological point of view. To get more details on biological background of the research, see Tyson and Hannsgen \cite{tysonhannsgen} or Murray and Hunt \cite{murrayhunt}.

In our paper we base on coupling methods introduced by Hairer in \cite{hairer}. In the same spirit, exponential rate of convergence was proven in \cite{sleczka} for classical iterated function systems (see also \cite{hairerm} or \cite{kapica}). However, we use coupling methods not only to evaluate the rate of convergence. It turns out that properly constructed coupling measure, if combined with the results for stationary ergodic Marokv chains given by Maxwell and Woodroofe \cite{woodr}, is crucial in the proof of the CLT, too. If we have the coupling measure already constructed, the proof of the CLT is brief and less technical than typical proofs based on Gordin's martingale approximation. What led us to this intriguing solution was an unsuccessful attempt to follow the pattern given by Komorowski and Walczuk \cite{komorowskiwalczuk}. It is worth mentioning here that an auxiliary model, described by some non-homogenous Markov chain, is needed to take adventage of coupling methods. %used for the deterministic case in \citet{hw}. 
While reading the paper, one may see that it is a~bright idea to express the Markov operator of interest by means of an auxiliary one.

Similar approach may also help to establish the Law of the Iterated Logarithm (LIL). The proof of the LIL is supposed to be provided in a future paper. Some ideas useful for proving it may be adapted from Bo\l{}t et al. \cite{bms}. However, we strongly believe that using an appropriate coupling measure will, again, make the proof much easier.

The organization of the paper goes as follows. Section $2$ introduces basic notation and definitions that are needed throughout the paper. Most of them are adapted from \cite{billingsley}, \cite{tweedie}, \cite{lasotay} and \cite{szarek}. Mathematical derivation of the generalized cell cycle model is provided in Section $3$. The main theorems (Theorem $\ref{main}$ and Theorem $\ref{CTG}$) are also formulated there. Sections $5$-$7$ are devoted to the construction of coupling measure for iterated function systems. Thanks to the results presented in Section $8$ we are finally able to present the proofs of main theorems. Indeed, the exponential rate of convergence is established in Section $9$ and the CLT in Section $10$.

\section{Notation and basic definitions}

Let $(X,\varrho)$ be a Polish space. We denote by $B_X$ the family of all Borel subsets of $X$. 
Let $B(X)$ be the space of all bounded and measurable functions $f:X\to R$ with the supremum norm and write $C(X)$ for its subspace of all bounded and continuous functions with the supremum norm. 

We denote by $M(X)$ the family of all Borel measures on $X$ and by $M_{\text{fin}}(X)$ and $M_1(X)$ its subfamilies such that $\mu(X)<\infty$ and~$\mu(X)=1$, respectively. 
Elements of $M_{\text{fin}}(X)$ which satisfy $\mu(X)\leq 1$ are called sub-probability measures. To simplify notation, we write
\[\langle f,\mu\rangle =\int_X f(x)\mu(dx)\quad\text{for}\: f:X\to R,\: \mu\in M(X).\]
An operator $P:M_{\text{fin}}(X)\to M_{\text{fin}}(X)$ is called a Markov operator if
\begin{enumerate}
\item $P(\lambda_1\mu_1+\lambda_2\mu_2)=\lambda_1 P\mu_1+\lambda_2 P\mu_2\quad\text{ for }\:\lambda_1,\lambda_2\geq 0,\; \mu_1, \mu_2\in M_{\text{fin}}(X)$;
\item $P\mu(X)=\mu(X)\quad\text{ for \:$\mu\in M_{\text{fin}}(X)$}$.
\end{enumerate}
Markov operator $P$ for which there exists a linear operator $U:B(X)\to B(X)$ such that
\[\langle Uf,\mu\rangle =\langle f,P\mu\rangle \quad\text{for}\: f\in B(X),\:\mu\in M_{\text{fin}}(X)\]
is called a regular operator. We say that a regular Markov operator is Feller if $U(C(X))\subset C(X)$. Every Markov operator $P$ may be extended to the space of signed measures on $X$ denoted by $M_{sig}(X)=\{\mu_1-\mu_2:\; \mu_1,\mu_2\in M_{\text{fin}}(X)\}$. For $\mu\in M_{sig}(X)$, we denote by $\|\mu\|$ the total variation norm of $\mu$, i.e.,
\[\|\mu\|=\mu^+(X)+\mu^-(X),\]
where $\mu^+$ and $\mu^-$ come from the Hahn-Jordan decomposition of $\mu$ (see \cite{halmos}). In particular, if $\mu$ is non-negative, $\|\mu\|$ is the total mass of $\mu$. 
For fixed $\bar{x}\in X$ we also consider the space $M_1^1(X)$ of all probability measures with finite first moment, i.e., $M_1^1(X)=\{\mu\in M_1(X):\:\int_X\varrho(x,\bar{x})\mu(dx)<\infty\}$. The family is independent of choice of $\bar{x}\in X$. 
We call $\mu_*\in M_{\text{fin}}(X)$ an invariant measure of $P$ if $P\mu_*=\mu_*$. 
For $\mu\in M_{\text{fin}}(X)$, we define the support of $\mu$ by
\begin{align*}
\mu=\{x\in X:\; \mu(B(x,r))>0\quad \text{for all }\:r>0\},
\end{align*}
where $B(x,r)$ is an open ball in $X$ with center at $x\in X$ and radius $r>0$. By $\bar{B}(x,r)$ we denote a~closed ball with center at $x\in X$ and radius $r>0$.

In $M_{sig}(X)$, we introduce the Fortet-Mourier norm
\[\|\mu\|_{\mathcal{L}}=\sup_{f\in\mathcal{L}}|\langle f,\mu\rangle |,\]
where $\mathcal{L}=\{f\in C(X):\;|f(x)-f(y)|\leq\varrho(x,y),\;|f(x)|\leq 1\;\text{ for }\:x,y\in X\}$. The space $M_1(X)$ with metric $\|\mu_1-\mu_2\|_{\mathcal{L}}$ is complete (see \cite{fortet}, \cite{rachev} or \cite{villani}).

We say that the sequence of Borel measures $(\mu_n)_{n\in N}\subset M_{fin}(X)$ converges weakly to the measure $\mu\in M_{fin}(X)$ if $\lim_{n\to \infty}\langle f,\mu_n\rangle=\langle f,\mu \rangle$ for all $f\in C(X)$. It is known (see Theorem 11.3.3, \cite{dudley}) that the following conditions are equivalent
\begin{itemize}
\item $(\mu_n)_{n\in N}$ converges weakly to $\mu$,
\item $\lim_{n\to\infty}\langle f,\mu_n\rangle =\langle f,\mu\rangle$ for all $f\in \mathcal{L}$,
\item $\lim_{n\to\infty}\|\mu_n-\mu\|_{\mathcal{L}}=0$,
\end{itemize}
where $(\mu_n)_{n\in N}\subset M_1(X)$ and $\mu\in M_1(X)$.

\section{Main idea and theorems}\label{sec:idea}

Recall that $(X,\varrho)$ is a Polish space and let $(\Omega,\mathcal{F},\text{Prob})$ be a probability space. 
Fix $T<\infty$. 
We consider a stochastically perturbed dynamical system. The state of $x_n$, for every $n\in N$, is determined by the formula
\[x_{n+1}=S(x_n,t_{n+1})\text{.}\]
We make the following assumptions. 

\begin{itemize}
\item[(I)] We consider a sequence $(t_n)_{n\in N}$ of independent random variables defined on $(\Omega,\mathcal{F},\text{Prob})$ with values in $[0,T]$. Distribution of $t_{n+1}$ conditional on $x_n=x$ is given by
\begin{equation}\label{distr_p}
\text{Prob}(t_{n+1}<t|x_n=x)=\int_0^tp(x,u)du, \quad 0\leq t\leq T,
\end{equation}
where $p:X\times[0,T]\to[0,\infty)$ is a~measurable and non-negative function. %such that, for every $x\in X$, $p(x,0)=0$ and $p(x,t)>0$ for $t>0$. 
In addition, $p$ is normalized, i.e., $\int_0^Tp(x,u)du=1$ for $x\in X$.

\item[(II)] Let $S:X\times[0,T]\to X$ be a~continuous function which satisfies the Lipschitz type inequality
\begin{equation}\label{Lip}
\varrho(S(x,t),S(y,t))\leq \lambda(x,t)\varrho(x,y)\quad\text{ for }x,y\in X, t\in[0,T],
\end{equation}
where $\lambda:X\times[0,T]\to[0,\infty)$ is a~Borel measurable function such that 
\begin{equation}\label{def:a}
a:=\sup_{x\in X}\int_0^T\lambda(x,t)p(x,t)dt<1.
\end{equation}

\item[(III)] $\sup_{t\in[0,T]}\varrho\Big(S(\bar{x},t),\bar{x}\Big)<\infty$ for some $\bar{x}\in X$.

\item[(IV)] We assume that $p$ satisfies the Dini condition
\begin{align}
\begin{aligned}
\int_0^T|p(x,t)-p(y,t)|dt\leq \omega(\varrho(x,y))\quad\text{for }x,y\in X,
\end{aligned}
\end{align}
where $\omega:{R}_+\to{R}_+$, $\omega(0)=0$, is a non-decreasing and concave function such that
\begin{align*}
\begin{aligned}
\int_0^{\sigma}\frac{\omega(t)}{t}dt<+\infty\quad\text{for some }\sigma> 0.
\end{aligned}
\end{align*}
We can easily check that if $\zeta<1$, we have
\begin{align}\label{varphi}
\begin{aligned}
\varphi(t):=\sum_{n=1}^{\infty}\omega(\zeta^nt)<+\infty\quad\text{for every }t\geq 0
\end{aligned}
\end{align}
and $\lim_{t\to 0}\varphi(t)=0$.

\item[(V)] Function $p$ is bounded. We set $\delta:=\inf_{x\in X, t\in(0,T]}p(x,t)$, $M:=\sup_{x\in X, t\in[0,T]}p(x,t)$ and require $\delta>0$.

\end{itemize}

We further assume that, for each $A\in B_X$,
\begin{align*}
\begin{aligned}
\text{Prob}(x_{n+1}\in A):=\mu_{n+1}(A) \quad\text{and}\quad P\mu_n=\mu_{n+1}\text{,}
\end{aligned}
\end{align*}
where
\begin{align*}
\begin{aligned}
P\mu(A)=\int_X\Big[\int_0^T 1_A(S(x,t))p(x,t)dt\Big]\mu(dx)\text{.}
\end{aligned}
\end{align*}
In \cite{lasotam} the proof of asymptotic stability is given for the model, while the exponential rate of convergence is established thanks to some coupling methods in \cite{hw}\footnote{In both papers the results are proven for stronger assumptions.}.

Without loss of generality, we may think of $(X,\varrho)$ as a closed subset of some separable Banach space $H$. Then, trying to describe some intercellular processes more precisely, Hille et al. \cite{hille} proposed a more general dynamical system
\[x_{n+1}=S(x_n,t_{n+1})+H_{n+1},\]
where $(H_n)_{n\in N}$, $H_n\in H$, is a family of independent random variables with the same distribution given by a measure $\nu^{\varepsilon}$, which is independent of $S(x_n,t_{n+1})$ and its support stays in $\bar{B}(0,\varepsilon)$. 

For this reason, we need an additional assumption
\begin{itemize}
\item[(VI)] Let $\varepsilon_*<\infty$ be given. Fix $\varepsilon\in[0,\varepsilon_*]$. Let $\nu^{\varepsilon}$ be a Borel measure on $H$ such that its support is in $\bar{B}(0,\varepsilon)$. For every $x\in X$, we set
\begin{align}
\begin{aligned}
\nu^{\varepsilon}_x(\cdot)=\nu^{\varepsilon}(\cdot-x).
\end{aligned}
\end{align}
We assume that $S(x,t)+h\in X$ for every $t\in[0,T]$, $x\in X$ and $h$ from the support of $\nu^{\varepsilon}$. 
\end{itemize}

The Markov chain is given by the transition function $\Pi_{\varepsilon}:X\times B_X\to[0,1]$ of the form
\[\Pi_{\varepsilon}(x,A)=\int_0^Tp(x,t)\nu^{\varepsilon}_{S(x,t)}(A)dt.\]
Then, we may write the Markov operator $P_{\varepsilon}:M_1(X)\to M_1(X)$ as follows
\[P_{\varepsilon}\mu(A)=\int_X\Pi_{\varepsilon}(x,A)\mu(dx).\]

The case of deterministic protein production, i.e., when $\varepsilon=0$, fits to the framework presented by Lasota and Mackey \cite{lasotam} and the results obtained there.

Hille et al. \cite{hille} managed to show the existence of a unique invariant measure in the generalized model, described above. However, stability was not proven. We want to focus on evaluating the rate of convergence, which additionally provides asymptotic stability in the model and allows us to establish the CLT. The proof of the CLT is given in Section 10.

\begin{thm}\label{main}
Let $\mu\in M_1^1(X)$. Under assumptions (I)-(VI), there exist $C=C(\mu)>0$ and $q\in[0,1)$ such that
\[\|P_{\varepsilon}^n\mu-\mu_*\|_{\mathcal{L}}\leq Cq^n\quad \text{ for }n\in N.\]
\end{thm}

Now, assumption (II) is strengthened to the following condition:
\begin{itemize}
\item[(II')] Let $S:X\times[0,T]\to X$ be a~continuous function which satisfies
\begin{equation}
\varrho(S(x,t),S(y,t))\leq \lambda(x,t)\varrho(x,y)\quad\text{ for }x,y\in X, t\in[0,T],
\end{equation}
where $\lambda:X\times[0,T]\to[0,\infty)$ is a~Borel measurable function such that 
\begin{equation}\label{def:b}
\Lambda:=\sup_{x\in X}\int_0^T\lambda^2(x,t)p(x,t)dt<1.
\end{equation}
\end{itemize}
Note that (II') implies (II), due to the H\"{o}lder inequality, and we obtain that $a\leq \sqrt{\Lambda}<1$. Assuming (II') instead of (II) allows us to show that $\mu_*\in M^2_1(X):=\{\mu\in M_1(X): \int_X\varrho^2(x,\bar{x})\mu(dx)<\infty\}$, which is essential to establish the CLT in the way presented in this paper. It is proven in Lemma \ref{second_moment} that $\mu_*$ is indeed with finite second moment.

Now, choose an arbitrary function $g:X\to R$ which is Lipschitz continuous, bounded and satisfies  $\langle g,\mu_*\rangle =0$. %Note that $\langle g^2,\mu_*\rangle  < \infty$. 
Let $(x_i)_{i\in N}$ be the Markov chain with transition probability function $\Pi_{\varepsilon}$ and initial distribution $\mu\in M^2_1(X)$. For every $n\in N$, put 
\begin{align*}
\begin{aligned}
\eta_n^{\mu}:=\frac{g(x_1)+\ldots+g(x_n)}{\sqrt{n}}
\end{aligned}
\end{align*}
and let $\Phi_{\eta_n^{\mu}}$ denote its distribution. %In particular, $\eta_n^*$ and $\eta_n^x$ are defined for the Markov chains with initial distributions $\mu_*$ and $\delta_x$, respectively. 

\begin{thm}\label{CTG}
Let $\mu\in M_1^2(X)$ be with finite second moment and let $\Phi_{\eta_n^{\mu}}$ be the distribution of $\eta_n^{\mu}$, as defined above. Assuming that all conditions (I)-(VI) are fulfilled and (II) is additionally strengthened to (II'). Then $\Phi_{\eta_n^{\mu}}$ converges weakly to the normal distribution, as $n\to\infty$. 
\end{thm}

\section{An auxiliary model - basic assumptions}\label{sec:assumptions}

Our aim is to prove exponential rate of convergence for the model given in \cite{hille}, as it is stated in Theorem $\ref{main}$. The idea is to use coupling methods. However, implementing these methods directly to the model given above does not give the expected results. Instead, we fix a sequence of constants $(h_n)_{n\in N}\subset H$, where $h_n\in\bar{B}(0,\varepsilon)$ for all $n\in N$, and consider a~stochastically perturbed dynamical system
\[x_{n+1}=T_{h_{n+1}}(x_n,t_{n+1}):=S(x_n,t_{n+1})+h_{n+1},\quad n=0,1,2,\ldots\]
Note that
\begin{align*}
\begin{aligned}
x_{n+1}&=T_{h_{n+1}}\Big(T_{h_{n}}(x_{n-1},t_n),t_{n+1}\Big).
\end{aligned}
\end{align*}
For abbreviation, we introduce the symbol
\begin{align}\label{def:tilde_x}
\tilde{x}_{n+1}^{x_0}:=T_{h_{n+1}}\Big(T_{h_n}(\ldots T_{h_2}(T_{h_1}(x_0,t_1),t_2)\ldots),t_{n+1}\Big),
\end{align}
where the upper index $x_0$ indicates the point from which the iteration begins. 

Let us further assume that, for every $A\in B_X$,
\[\text{Prob}(x_{n+1}\in A):=\mu_{n+1}(A) \quad\text{and}\quad P_{h_{n+1}}\mu_n=\mu_{n+1}\text{,}\]
where, for arbitrary $h\in\bar{B}(0,\varepsilon)$,
\[(P_{h}\mu)(A):=\int_X\left[\int_0^T 1_A(T_{h}(x,t))p(x,t)dt\right]\mu(dx)\text{.}
\]

We maintain all previous assumptions (I)-(VI). Now, for every $h\in\bar{B}(0,\varepsilon)$, we consider an operator
\[
T_h(x,t):=S(x,t)+h.
\]
Note that, as a consequence of assumption (II), $T_h$ is continuous and satisfies the same Lipschitz type inequality as operator $S$ satisfies.

We also set 
\begin{align}\label{def:c}
\begin{aligned}
c:=\sup_{t\in[0,T]}\varrho\Big(S(\bar{x},t),\bar{x}\Big)+\varepsilon_*>\sup_{t\in[0,T], i\in N}\varrho\Big(T_{h_i}(\bar{x},t),\bar{x}\Big).
\end{aligned}
\end{align}
Obviously, $c$ is finite, because of assumption (III).

\section{Measures on the pathspace and coupling}\label{sec:coupling}

Set $x\in X$ and $(h_n)_{n\in N}\subset H$, $h_n\in\bar{B}(0,\varepsilon)$ for all $n\in N$. One-dimensional distributions $\Pi^n_{h_1,\ldots,h_n}(x,\cdot)$, ${n\in N}$, are defined by induction on $n$
\begin{align}\label{constr:n+1}
\begin{aligned}
\Pi^0(x,A)=\delta_x(A) \\
\Pi^1_{h}(x,A)=\int_0^T1_A(T_{h}(x,t))p(x,t)dt \\
\vdots \\ 
\Pi^n_{h_1,\ldots,h_n}(x,A)=\int_X \Pi^1_{h_n}(y,A)\Pi^{n-1}_{h_1,\ldots,h_{n-1}}(x,dy),
\end{aligned}
\end{align}
where $A\in B_X$. 
We easily obtain two-dimensional and higher-dimensional distributions. If we assume that, for $x\in X$, $\Pi^{1,\ldots,n}_{h_1,\ldots,h_n}(x,\cdot)$ is a measure on $X^n$, generated by a sequence $(\Pi^1_{h_i}(x,\cdot))_{i=1}^n$, then 
\begin{align}\label{constr:ndim}
\begin{aligned}
\Pi^{1,\ldots,n+1}_{h_1,\ldots,h_{n+1}}(x,A\times B)=\int_A\Pi^1_{h_{n+1}}(z_n,B)\Pi^{1,\ldots,n}_{h_1,\ldots,h_n}(x,dz),
\end{aligned}
\end{align}
where $z=(z_1,\ldots,z_n)$ and $A\in B_{X^n}$, $B\in B_X$, is a measure on $X^{n+1}$. 
Note that $\Pi^1_{h_1}(x,\cdot),\ldots,\Pi^n_{h_1,\ldots,h_n}(x,\cdot)$, given by $(\ref{constr:n+1})$, are marginal distributions of $\Pi^{1,\ldots,n}_{h_1,\ldots,h_n}(x,\cdot)$, for every $x\in X$. 
Finally, we obtain a~family $\{\Pi_{h_1,h_2,\ldots}^{\infty}(x,\cdot):x\in X\}$ of sub-probability measures on $X^{\infty}$. This construction is motivated by \cite{hairer}. The existence of measures $\Pi_{h_1,h_2,\ldots}^{\infty}(x,\cdot)$ is established by the Kolmogorov theorem. More precisely, for any $x\in X$, there exists some probability space on which we can define a~stochastic process $\xi^x$ with distribution $\phi_{\xi^x}$ such that
\[\phi_{\xi^x}(B)=\text{Prob}(\xi^x\in B):=\Pi_{h_1,h_2,\ldots}^{\infty}(x,B)\quad\text{ for $B\in B_{X^{\infty}}$.}\] 
Therefore, $\Pi_{h_1,h_2,\ldots}^{\infty}(x,\cdot)$ is the distribution of the non-homogeneous Markov chain $\xi^x$ on $X^{\infty}$ with sequence of transition probability functions $(\Pi^1_{h_i})_{i\in N}$  and $\phi_{\xi^x_0}=\delta_x$, for $x\in X$. If an initial distribution is given by any $\mu\in M_{\text{fin}}(X)$, not necessarily by $\delta_x$, we define
\[(P_{h_1,h_2,\ldots}^{\infty}{\mu})(B)=\int_X \Pi_{h_1,h_2,\ldots}^{\infty}(x,B)\mu(dx)\quad\text{ for $B\in B_{X^{\infty}}$}.\]

\begin{definition}
Let a family of probability measures $(\{\Pi^1_{h_i}(x,\cdot): x\in X\})_{i\in N}$ be given. For every $i\in N$, we can set another family of probability measures $\{C^1_{h_i}((x,y),\cdot):x,y\in X\}$ on $X^2$ such that
\begin{itemize}
\item $C^1_{h_i}((x,y),A\times X)=\Pi^1_{h_i}(x,A)$ \;for $A\in B_X$,
\item $C^1_{h_i}((x,y),X\times B)=\Pi^1_{h_i}(y,B)$ \;for $B\in B_X$,
\end{itemize}
where $x,y\in X$. For every $i\in N$, $\{C^1_{h_i}((x,y),\cdot):x,y\in X\}$ is called coupling.
\end{definition}

\section{Iterated function systems}\label{sec:ifs}

We consider a continuous function $S:X\times[0,T]\to X$ and a sequence of continuous mappings given by $(T_{h_i})_{i\in N}$ with sequence of constants $(h_i)_{i\in N}$ established. We assume that $p:X\times[0,T]\to[0,\infty)$ is a  non-negative and normalized function.  
For each $A\in B_X$, we build a~sequence of transition operators, as we did in $(\ref{constr:n+1})$.

Let $n\in N$. Note that, for arbitrary $A\in B_X$, $\Pi^n_{h_1,\ldots,h_n}(\cdot,A):X\to R$ is measurable by definition. 
Furthermore, $\Pi^n_{h_1,\ldots,h_n}(x,\cdot):B_X\to R$ is a probability measure, for every $x\in X$. Hence, $\Pi^n_{h_1\ldots,h_n}$ is a transition probability function on the $n$-th marginal. 
Thanks to these properties (see Section 1.1, \cite{z}), for every $n\in N$ and a sequence of constants $(h_i)_{i\in N}$ fixed, there exists a~unique regular Markov operator $P^n_{h_1,\ldots,h_n}$, for which $\Pi^n_{h_1,\ldots,h_n}$ is a transition probability function, and it is given by the formula
\[(P^n_{h_1,\ldots,h_n}\mu)(A)=\int_X\Pi^n_{h_1,\ldots,h_n}(x,A)\mu(dx),\]
where $A\in B_X$, $\mu\in M_{1}(X)$. Moreover, a dual operator $U^n_{h_1,\ldots,h_n}:B(X)\to B(X)$ to $P^n_{h_1,\ldots,h_n}$ is defined as follows
\[(U^n_{h_1\ldots,h_n}f)(x)=\int_Xf(y)\Pi^n_{h_1,\ldots,h_n}(x,dy).\]

\begin{remark}
According to assumptions (II) and (IV), one may check, although through some tedious computations, that, for every $n\in N$ and a sequence of constants $(h_i)_{i\in N}$ fixed,
\begin{align*}
\begin{aligned}
\|\Pi^n_{h_1,\ldots,h_n}(x,\cdot)-\Pi^n_{h_1,\ldots,h_n}(y,\cdot)\|_{\mathcal{L}}\leq a^n\varrho(x,y)+\varphi(\varrho(x,y)),
\end{aligned}
\end{align*}
where $\varphi$ is given by (\ref{varphi}). This indicates weak continuity of the map $X\ni x\mapsto\Pi^n_{h_1,\ldots,h_n}(x,\cdot)\in M_1(X)$.
Now, this property, together with the fact that $P^n_{h_1,\ldots,h_n}$ is a regular Markov operator, impies that $P^n_{h_1,\ldots,h_n}$ is even Feller (see Chapter 6, \cite{tweedie}).

However, these estimates do not give us any proper result about stability of the model. That is why we still need to use some coupling methods.
\end{remark}

Repeating the construction from the previous section, we obtain $P_{h_1,h_2,\ldots}^{\infty}\mu$ for $\mu\in M_{1}(X)$. Obviously, for every $n\in N$, $P^n_{h_1,\ldots,h_n}{\mu}$ is the $n$-th marginal of $P_{h_1,h_2,\ldots}^{\infty}\mu$.

Fix $\bar{x}\in X$ for which assumption (III) holds. We define $V:X\to[0,\infty)$ to be
\[V(x)=\varrho(x,\bar{x}).\]

\begin{lemma}
For every $n\in N$ and a sequence of constants $(h_i)_{i\in N}$ fixed, if $\mu\in M^1_1(X)$, then $P^n_{h_1,\ldots,h_n}\mu\in M^1_1(X)$. Moreover, 
\[\langle V,P^n_{h_1,\ldots,h_n}\mu\rangle \leq a^n\langle V,\mu\rangle +\frac{1}{1-a}c,\]
where $c$ does not depend on the sequence $(h_i)_{i\in N}$.
\end{lemma}

\begin{proof}[Proof]
Recall that $a<1$ and $c$ are given by (\ref{def:a}) i (\ref{def:c}), respectively. The state $\tilde{x}^x_n$ is of the form (\ref{def:tilde_x}).
Following $(\ref{Lip})$, we obtain
\begin{align*}
\begin{aligned}
&\langle V,P^n_{h_1,\ldots,h_n}\mu \rangle  \\
&=\int_X\Big[\int_0^T\ldots\int_0^T\varrho(\tilde{x}^x_n,\bar{x})p(\tilde{x}^x_{n-1},t_n)p(\tilde{x}^x_{n-2},t_{n-1})\ldots p(x,t_1)dt_n\ldots dt_1\Big]\mu(dx)\\
&\leq\int_X\int_0^T\ldots\int_0^T\Big[\varrho(\tilde{x}^x_n,\tilde{x}^{\bar{x}}_n)+\varrho(\tilde{x}^{\bar{x}}_n,\bar{x})\Big]
p(\tilde{x}^x_{n-1},t_n)p(\tilde{x}^x_{n-2},t_{n-1})\ldots p(x,t_1)dt_n\ldots dt_1\mu(dx)\\
&\leq \int_X\int_0^T\ldots\int_0^T \Big[ \lambda(\tilde{x}^x_{n-1},t_n)\lambda(\tilde{x}^x_{n-2},t_{n-1})\ldots\lambda(x,t_1)\varrho(x,\bar{x})+\varrho(\tilde{x}^x_{n},\tilde{x}^{\bar{x}}_{n-1})+\ldots+\varrho(\tilde{x}^{\bar{x}}_{1},\bar{x})\Big]\\
&\qquad p(\tilde{x}^x_{n-1},t_n)p(\tilde{x}^x_{n-2},t_{n-1})\ldots p(x,t_1)dt_n\ldots dt_1\mu(dx)\\
&\leq \int_X a^n\varrho(x,\bar{x})+c(a^n+\ldots+1)\mu(dx)\\
&\leq a^n\langle V,\mu\rangle +\frac{c}{1-a},
\end{aligned}
\end{align*}
which completes the proof.
\end{proof}

Fix probability measures $\mu,\nu\in M_1^1(X)$ and Borel sets $A,B\in B_X$. We consider $b\in M_{\text{fin}}(X^2)$ such that
\[b(A\times X)=\mu(A)\text{,}\qquad b(X\times B)=\nu(B)\]
and $b^n_{h_1,\ldots,h_n}\in M_{\text{fin}}(X^2)$ such that, for every $n\in N$,
\[b^n_{h_1,\ldots,h_n}(A\times X)=(P^n_{h_1,\ldots,h_n}\mu)(A)\text{,}\qquad b^n_{h_1,\ldots,h_n}(X\times B)=(P^n_{h_1,\ldots,h_n}\nu)(B)\text{.}\]
Furthermore, we define $\bar{V}:X^2\to[0,\infty)$
\[\bar{V}(x,y)=V(x)+V(y)\quad\text{for $x,y\in X$.}\]
Note that, for every $n\in N$,
\begin{align}\label{prop:barV}
\langle \bar{V},b^n_{h_1,\ldots,h_n}\rangle \:\leq \:a\langle \bar{V},b^{n-1}_{h_1,\ldots,h_{n-1}}\rangle +2c\:\leq\:a^n\langle \bar{V},b\rangle +\frac{2}{1-a}c.
\end{align}
For measures $b\in M_{\text{fin}}^1(X^2)$ finite on $X^2$ and with finite first moment, we define the~linear functional
\[\phi(b)=\int_{X^2}\varrho(x,y)b(dx\times dy).\]
Following the above definitions, we easily obtain 
\begin{align}\label{prop:phi}
\phi(b)\:\leq\:\langle \bar{V},b\rangle .
\end{align}

\section{Coupling for iterated function systems}\label{sec:coupling_ifs}

On $X^2$ we define the transition sub-probability functions such that, for $A,B\in B_{X}$,
\begin{equation}\label{def:Qxy}
Q^1_{h_i}((x,y),A\times B)=\int_0^T \min\{p(x,t),p(y,t)\}\delta_{(T_{h_i}(x,t),T_{h_i}(y,t))}(A\times B)dt,\quad i\in N,
\end{equation}
and
\begin{equation}\label{def:Qxy^n}
Q^n_{h_1,\ldots,h_n}((x,y),A\times B)=\int_{X^2}Q^1_{h_n}((u,v),A\times B)Q^{n-1}_{h_1,\ldots,h_{n-1}}((x,y),du\times dv),\quad n\in N.
\end{equation}
Measures generated by the transition functions defined above are, by convention, denoted with the same letter. Every time, the context should indicate what we mean. 
It is easy to check that, for every $i\in N$,
\begin{align*}
\begin{aligned}
Q^1_{h_i}((x,y),A\times X)&\leq\int_0^Tp(x,t)\delta_{T_{h_i}(x,t)}(A)dt
&=\int_0^T1_A(T_{h_i}(x,t))p(x,t)dt
&=\Pi^1_{h_i}(x,A)
\end{aligned}
\end{align*}
and analogously 
$Q^1_{h_i}((x,y),X\times B)\leq \Pi^1_{h_i}(y,B)$. 
Similarly, for $n\in N$,
\[Q^n_{h_1,\ldots,h_n}((x,y),A\times X)\leq\Pi^n_{h_1,\ldots,h_n}(x,A)\text{,}\qquad Q^n_{h_1,\ldots,h_n}((x,y),X\times B)\leq\Pi^n_{h_1,\ldots,h_n}(y,B).\]
For $b\in M_{\text{fin}}(X^2)$, let $Q^n_{h_1,\ldots,h_n}b$ denote the measure
\begin{equation}\label{def:Qb}
(Q^n_{h_1,\ldots,h_n}b)(A\times B)=\int_{X^2}Q^n_{h_1,\ldots,h_n}((x,y),A\times B)b(dx\times dy)\quad\text{ for }\:A,B\in B_{X}, n\in N\text{.}
\end{equation}
Note that, for every $A,B\in B_{X}$ and $n\in N$, we obtain
\begin{align}\label{Q1Qnb}
\begin{aligned}
(Q^{n+1}_{h_1,\ldots,h_{n+1}}b)(A\times B)
&=\int_{X^2}Q_{h_1,\ldots,h_{n+1}}^{n+1}((x,y),A\times B)b(dx\times dy)\\
&=\int_{X^2}\int_{X^2}Q^1_{h_{n+1}}((u,v),A\times B)Q^n_{h_1,\ldots,h_n}((x,y),du\times dv)b(dx\times dy)\\
&=\int_{X^2}Q^1_{h_{n+1}}((u,v),A\times B)(Q^n_{h_1,\ldots,h_n}b)(du\times dv)
=(Q^1_{h_{n+1}}(Q_{h_1,\ldots,h_n}^nb))(A\times B).
\end{aligned}
\end{align}
Again, following $(\ref{constr:n+1})$ and $(\ref{constr:ndim})$, we are able to construct measures on products and, as a consequence, a measure $Q_{h_1,h_2,\ldots}^{\infty}b$ on~$X^{\infty}$, for every $b\in M_{\text{fin}}(X^2)$. 
Now, we check that, for $n\in N$ and $b\in M^1_{\text{fin}}(X^2)$,
\begin{equation}\label{prop:phiQb}
\phi(Q^n_{h_1,\ldots,h_n}b)\leq a^n\phi(b).
\end{equation}
Let us observe that
\begin{align*}
\begin{aligned}
\phi(Q^n_{h_1,\ldots,h_n}b)
&=\int_{X^2}\int_{X^2}\varrho(u,v)Q^n_{h_1,\ldots,h_n}((x,y),du\times dv)b(dx\times dy)\\
&=\int_{X^2}\int_{X^2}\int_0^T\int_{X^2}\varrho(u,v)\min\{p(\bar{u},t),p(\bar{v},t)\}\delta_{(T_{h_n}(\bar{u},t),T_{h_n}(\bar{v},t))}(du\times dv)dt\\
&\qquad\qquad\qquad Q^{n-1}_{h_1,\ldots,h_{n-1}}((x,y),d\bar{u}\times d\bar{v})b(dx\times dy)\\
&\leq\int_{X^2}\int_{X^2}\int_0^T\varrho(T_{h_n}(\bar{u},t),T_{h_n}(\bar{v},t))p(\bar{u},t)dtQ^{n-1}_{h_1,\ldots,h_{n-1}}((x,y),d\bar{u}\times d\bar{v})b(dx\times dy)\\
&\leq\int_{X^2}\int_{X^2}\int_0^T\varrho(\bar{u},\bar{v})\lambda(\bar{u},t)p(\bar{u},t)dtQ^{n-1}_{h_1,\ldots,h_{n-1}}((x,y),d\bar{u}\times d\bar{v})b(dx\times dy)\\
&\leq a\int_{X^2}\int_{X^2}\varrho(\bar{u},\bar{v})Q^{n-1}_{h_1,\ldots,h_{n-1}}((x,y),d\bar{u}\times d\bar{v})b(dx\times dy)\leq\ldots\leq a^n\phi(b).
\end{aligned}
\end{align*}

For every $i\in N$, we can find a measure $R^1_{h_i}((x,y),\cdot)$ such that the sum of $Q^1_{h_i}((x,y),\cdot)$ and $R^1_{h_i}((x,y),\cdot)$ gives a new coupling measure $C^1_{h_i}((x,y),\cdot)$.

\begin{lemma}\label{lemma1}
Fix $i\in N$. There exists the family $\{R^1_{h_i}((x,y),\cdot):x,y\in X\}$ of measures on~$X^2$ such that we can define
\[C^1_{h_i}((x,y),\cdot)=Q^1_{h_i}((x,y),\cdot)+R^1_{h_i}((x,y),\cdot)\quad\text{ for $x,y\in X$}\]
and, moreover, 
\begin{itemize}
\item[(i)] the mapping $(x,y)\mapsto R^1_{h_i}((x,y),A\times B)$ is measurable for every $A,B\in B_{X}$;
\item[(ii)] measures $R^1_{h_i}((x,y),\cdot)$ are non-negative for $x,y\in X$;
\item[(iii)] measures $C^1_{h_i}((x,y),\cdot)$ are probabilistic for every $x,y\in X$ and so $\{C^1_{h_i}((x,y),\cdot):x,y\in X\}$ is a~transition probability function on $X^2$;
\item[(iv)] for every $A,B\in B_X$ and $x,y\in X$, we get $C^1_{h_i}((x,y),A\times X)=\Pi^1_{h_i}(x,A)$ and $C^1_{h_i}((x,y),X\times B)=\Pi^1_{h_i}(y,B)$.
\end{itemize}
\end{lemma}

\begin{proof}[Proof]
Fix $A,B\in B_{X}$. Let
\begin{align*}
\begin{aligned}
&R^1_{h_i}((x,y),A\times B)\\
&\qquad =
(1-Q^1_{h_i}((x,y),X^2))^{-1}(\Pi^1_{h_i}(x,A)-Q^1_{h_i}((x,y),A\times X))(\Pi^1_{h_i}(y,B)-Q^1_{h_i}((x,y),X\times B))
\end{aligned}
\end{align*}
if $Q^1_{h_i}((x,y),X^2)<1$ and 
$R^1_{h_i}((x,y),A\times B)=0$ if $Q^1_{h_i}((x,y),X^2)=1$. 
Obviously, the formula may be extended to the measure. 
The mapping has all desirable properties $(i)-(iv)$.
\end{proof}

Lemma $\ref{lemma1}$ shows that, for every $i\in N$, we may construct the coupling $\{C^1_{h_i}((x,y),\cdot):x,y\in X\}$ for $\{\Pi^1_{h_i}(x,\cdot):x\in X\}$ such that $Q^1_{h_i}((x,y),\cdot)\leq C^1_{h_i}((x,y),\cdot)$, whereas measures $R^1_{h_i}((x,y),\cdot)$ are non-negative. Following the rules given in $(\ref{constr:n+1})$, $(\ref{constr:ndim})$, as well as the whole construction from Section $\ref{sec:coupling}$, we easily obtain the family of probability measures $\{C^{\infty}_{h_1,h_2\ldots}((x,y),\cdot):x,y\in X\}$ on $(X^2)^{\infty}$ with marginals $\Pi_{h_1,h_2,\ldots}^{\infty}(x,\cdot)$ and $\Pi_{h_1,h_2,\ldots}^{\infty}(y,\cdot)$. This construction appears in \cite{hairer}. %We may also consider a sequence of distributions $(\{C^n_{h_1,\ldots,h_n}((x,y),\cdot):x,y\in X\})_{n\in N}$, constructed by induction on $n$, as it is done in $(\ref{constr:n+1})$. 
Note that, for every $n\in N$ and $x,y\in X$, $C^n_{h_1,\ldots,h_n}((x,y),\cdot)$, constructed as in $(\ref{constr:n+1})$, is the $n$-th marginal of $C^{\infty}_{h_1,h_2\ldots}((x,y),\cdot)$. Additionally, $\{C^n_{h_1,\ldots,h_n}((x,y),\cdot):x,y\in X\}$ fulfills the role of coupling for $\{\Pi^n_{h_1,\ldots,h_n}(x,\cdot):x\in X\}$. Indeed, for $A\in B_X$,
\begin{align*}
\begin{aligned}
C^n_{h_1,\ldots,h_n}((x,y),A\times X)&=\int_{X^2}C^1_{h_n}((u,v),A\times X)C^{n-1}_{h_1,\ldots,h_{n-1}}((x,y),du\times dv)\\
&=\int_{X^2}\Pi^1_{h_n}(u,A)C^{n-1}_{h_1,\ldots,h_{n-1}}((x,y),du\times dv)=\ldots=\Pi^n_{h_1,\ldots,h_n}(x,A)
\end{aligned}
\end{align*}
and, similarly, $C^n_{h_1,\ldots,h_n}((x,y),X\times B)=\Pi^n_{h_1,\ldots,h_n}(y,B)$.

Fix $(x_0,y_0)\in X^2$ and $(h_n)_{n\in N}\subset[0,\varepsilon)$. The sequence of transition probability functions $\Big(\{C^n_{h_1,\ldots,h_n}((x,y),\cdot):x,y\in X\}\Big)_{n\in N}$ defines the non-homogenous Markov chain $\Psi$ on $X^2$ with starting point $(x_0,y_0)$, while the sequence of transition probability functions $\Big(\{\hat{C}^n_{h_1,\ldots,h_n}((x,y,\theta),\cdot):x,y\in X, \theta\in\{0,1\}\}\Big)_{n\in N}$ defines the Markov chain $\hat{\Psi}$ on the augmented space $X^2\times\{0,1\}$ with initial distribution $\hat{C}^0((x_0,y_0),\cdot)=\delta_{(x_0,y_0,1)}(\cdot)$. If $\hat{\Psi}_n=(x,y,i)$,
where $x,y\in X$, $i\in\{0,1\}$, then
\[\text{Prob}(\hat{\Psi}_{n+1}\in A\times B\times\{1\}\:|\:\hat{\Psi}_n=(x,y,i),i\in\{0,1\})=Q^n_{h_1,\ldots,h_n}((x,y),A\times B),\]
\[\text{Prob}(\hat{\Psi}_{n+1}\in A\times B\times\{0\}\:|\:\hat{\Psi}_n=(x,y,i),i\in\{0,1\})=R^n_{h_1,\ldots,h_n}((x,y),A\times B),\]
where $A,B\in B_X$. 
Once again, we refer to $(\ref{constr:n+1})$, $(\ref{constr:ndim})$ and the Kolmogorov theorem to obtain the measure $\hat{C}^{\infty}_{h_1,h_2,\ldots}((x_0,y_0),\cdot)$ on $(X^2\times\{0,1\})^{\infty}$ which is associated with the Markov chain $\hat{\Psi}$.

From now on, we assume that processes $\Psi$ and $\hat{\Psi}$ taking values in $X^2$ and $X^2\times\{0,1\}$, respectively, are defined on $(\Omega, F, \mathbf{P})$. The expected value of measures $C^{\infty}_{h_1,h_2,\ldots}((x_0,y_0),\cdot)$, $\hat{C}^{\infty}_{h_1,h_2,\ldots}((x_0,y_0),\cdot)$ is denoted by $E_{x_0,y_0}$.

\section{Auxiliary theorems}

Recall that $a$ is given by $(\ref{def:a})$. Fix $\varkappa\in(0,1-a)$. Set
\[K_{\varkappa}=\{(x,y)\in X^2: \: \bar{V}(x,y)<\varkappa^{-1}2c\},\]
where $c$ is given by $(\ref{def:c})$. Let $d:(X^2)^{\infty}\to N$ denote the time of the first visit in $K_{\varkappa}$, i.e.
\[d((x_n,y_n)_{n\in N})=\inf\{n\in N:\: (x_n,y_n)\in K_{\varkappa}\}.\]
As a convention, we put $d((x_n,y_n)_{n\in N})=\infty$, if there is no $n\in N$ such that $(x_n,y_n)\in K_{\varkappa}$.

\begin{thm}\label{theorem1}
For every $\zeta\in(0,1)$ there exist positive constants $C_1,C_2$ such that
\[E_{x_0,y_0}[(a+\varkappa)^{-\zeta d}]\leq C_1\bar{V}(x_0,y_0)+C_2.\]
\end{thm} 

\begin{proof}[Proof]
Fix $(x_0,y_0)\in X^2$. Let $\Psi=(x_n,y_n)_{n\in N}$ be the Markov chain with starting point $(x_0,y_0)$ and sequence of transition probability functions $\Big(\{C^1_{h_i}((x,y),\cdot):x,y\in X\}\Big)_{i\in N}$. Let $F_n\subset F$, $n\in N$, be the natural filtration in $\Omega$ associated with $\Psi$. We define
\[A_n=\{\omega\in\Omega: \; \Psi_i=(x_i(\omega), y_i(\omega))\notin K_{\varkappa}\; \text{ for }\:i=1,\ldots,n\}, \quad n\in N.\]
Obviously, $A_{n+1}\subset A_n$ and $A_n\in F_n$, for $n\in N$. 
In consequence of $(\ref{prop:barV})$, as well as the definitions of $A_n$ and $K_{\varkappa}$, the following inequalities are $\mathbf{P}$-a.s. satisfied in $\Omega$:
\[1_{A_n}E_{x_0,y_0}[\bar{V}(x_{n+1},y_{n+1})|F_n]\leq 1_{A_n}(a\bar{V}(x_n,y_n)+2c)\leq 	
1_{A_n}(a+\varkappa)\bar{V}(x_n,y_n).\]Accordingly, we obtain
\begin{align*}
\begin{aligned}
\int_{A_n}\bar{V}(x_n,y_n)d\mathbf{P}&\leq\int_{A_{n-1}}\bar{V}(x_n,y_n)d\mathbf{P}
=\int_{A_{n-1}}E[\bar{V}(x_n,y_n)|F_{n-1}]d\mathbf{P}\\
&\leq\int_{A_{n-1}}[a\bar{V}(x_{n-1},y_{n-1})+2c]d\mathbf{P}
\leq(a+\varkappa)\int_{A_{n-1}}\bar{V}(x_{n-1},y_{n-1})d\mathbf{P}.
\end{aligned}
\end{align*}
On applying these estimates finitely many times, we obtain
\[\int_{A_n}\bar{V}(x_n,y_n)d\mathbf{P}\leq(a+\varkappa)^{n-1}\int_{A_1}\bar{V}(x_1,y_1)d\mathbf{P}\leq(a+\varkappa)^{n-1}[a\bar{V}(x_0,y_0)+2c].\]
Note that
\begin{align*}
\begin{aligned}
\mathbf{P}(A_n)&\leq\int_{A_n}\varkappa(2c)^{-1}\bar{V}(x_n,y_n)d\mathbf{P}
&\leq\varkappa[2c(a+\varkappa)]^{-1}(a+\varkappa)^n[a\bar{V}(x_0,y_0)+2c].
\end{aligned}
\end{align*}
Set $\hat{c}:=\varkappa[2c(a+\varkappa)]^{-1}[a\bar{V}(x_0,y_0)+2c]$. Then, $\mathbf{P}(A_n)\leq(a+\varkappa)^n\hat{c}$. 
Fix $\zeta\in(0,1)$. Since $d$ takes natural values $n\in N$, we obtain
\begin{align*}
\begin{aligned}
\sum_{n=1}^{\infty}(a+\varkappa)^{-\zeta n}\mathbf{P}(A_n)
&\leq\sum_{n=1}^{\infty}(a+\varkappa)^{-\zeta n}(a+\varkappa)^n\hat{c}
&=\sum_{n=1}^{\infty}(a+\varkappa)^{(1-\zeta)n}\hat{c},\\
\end{aligned}
\end{align*}
which implies convergence of the series. The proof is complete by the definition of $\hat{c}$ and with properly choosen $C_1$, $C_2$.
\end{proof}
For every positive $r>0$, we define the set
\[C_r=\{(x,y)\in X^2:\; \varrho(x,y)<r\}.\]

\begin{lemma}\label{lemma2}
Fix $\tilde{a}\in(a,1)$. Let $C_r$ be the set defined above and suppose that $b\in M_{\text{fin}}(X^2)$ is such that $\text{supp }b\subset C_r$. There exists $\bar{\gamma}>0$ such that
\begin{align*}
(Q^n_{h_1,\ldots,h_n}b)(C_{\tilde{a}^nr})\geq\bar{\gamma}^n\|b\|
\end{align*}
for $\delta$ and $M$ defined in assumption (V) (see Section $\ref{sec:idea}$).
\end{lemma}

\begin{proof}[Proof]
Recall that $\tilde{x}^x_n$ is given by (\ref{def:tilde_x}). Directly from $(\ref{def:Qb})$, $(\ref{def:Qxy^n})$ and (\ref{def:Qxy}) we obtain
\begin{align*}
\begin{aligned}
&(Q^n_{h_1,\ldots,h_n}b)(C_{\tilde{a}^nr})\\
&=\int_{X^2}\int_{X^2}\int_0^T\min\{p(u,t_n),p(v,t_n)\}
\delta_{(T_{h_n}(u,t_n),T_{h_n}(v,t_n))}(C_{\tilde{a}^nr})dt_n Q^{n-1}_{h_1,\ldots,h_{n-1}}((x,y),du\times dv) \: b(dx\times dy)\\
&=\int_{X^2}\Big[\int_{(0,T)^n}1_{C_{\tilde{a}^nr}}(\tilde{x}^x_n,\tilde{x}^y_n)
\min\{p(\tilde{x}^x_{n-1},t_n),p(\tilde{x}^y_{n-1},t_n)\}\ldots \min\{p(x,t_1),p(y,t_1)\}
dt_n\ldots dt_1\Big]b(dx\times dy).\\
\end{aligned}
\end{align*}
Note that $1_{C_{\tilde{a}^nr}}(\tilde{x}^x_n,\tilde{x}^y_n)=1$ if and only if $(t_1,\ldots,t_n)\in\mathcal{T}_n$, where 
\[\mathcal{T}_n:=\{(t_1,\ldots,t_n)\in(0,T)^n:\varrho(\tilde{x}^x_n,\tilde{x}^y_n)<\tilde{a}^nr\}.\]
Set $\mathcal{T}_n':=(0,T)^n\backslash\mathcal{T}_n$. 
Note that, according to assumption (II), we have
\begin{align*}
\begin{aligned}
&\int_{\mathcal{T}_n'}
\varrho(\tilde{x}^x_n,\tilde{x}^y_n)
p(\tilde{x}^x_{n-1},t_n)\ldots p(x,t_1)
dt_n\ldots dt_1 \leq a^n\varrho(x,y)< a^nr
\end{aligned}
\end{align*}
for $(x,y)\in C_r$. 
Comparing this with the definition of $\mathcal{T}_n'$, we obtain
\begin{align*}
\begin{aligned}
&\tilde{a}^n r \int_{\mathcal{T}_n'}p(\tilde{x}^x_{n-1},t_n)
\ldots p(x,t_1)dt_n\ldots dt_1< a^nr,
\end{aligned}
\end{align*}
which implies
\begin{align*}
\begin{aligned}
\int_{\mathcal{T}_n'}p(\tilde{x}^x_{n-1},t_n)
\ldots p(x,t_1)dt_n\ldots dt_1
<\frac{a^n}{\tilde{a}^n}<1.
\end{aligned}
\end{align*}
We then obtain that the integral over $\mathcal{T}_n$ is not less than $1-\Big(\frac{a}{\tilde{a}}\Big)^n\geq (1-\frac{a}{\tilde{a}})^n=:\gamma^n$, for sufficiently big $n\in N$, which provides, using assumption (V), that $|\mathcal{T}_n|\geq \Big(\frac{\gamma}{M}\Big)^n$. Finally,
\begin{align*}
\begin{aligned}
(Q^n_{h_1,\ldots,h_n}b)(C_{\tilde{a}^nr})&\geq\int_{X^2}\delta^n|\mathcal{T}_n|b(dx,dy)
\geq\delta^n\Big(\frac{\gamma}{M}\Big)^n\|b\|.
\end{aligned}
\end{align*}
If we set $\bar{\gamma}:=\delta M^{-1}\gamma$, the proof is complete.
\end{proof}

\begin{thm}\label{theorem2}
For every $\varkappa\in(0,1-a)$, there exists $n_0\in N$ such that
\[\|Q_{h_1,h_2,\ldots}^{\infty}((x,y),\cdot)\|\geq\frac{1}{2}\bar{\gamma}^{n_0}\quad\text{for $(x,y)\in K_{\varkappa}$,}\]
where $\bar{\gamma}>0$ is given in Lemma \ref{lemma2}.
\end{thm}

\begin{proof}[Proof]
Note that, for every real numbers $u, v\in R$, there is a general rule: $\;\min\{u,v\}+|u-v|-u\geq 0$. Hence, for every $(x,y)\in X^2$ and $i\in N$, we obtain
\begin{align*}
\begin{aligned}
\int_0^T\Big[\min\{p(x,t),p(y,t)\}+|p(x,t)-p(y,t)|-p(x,t)\Big]dt\geq 0
\end{aligned}
\end{align*} 
and therefore, due to (\ref{def:Qxy}), 
\begin{align*}
\begin{aligned}
\|Q_{h_i}^1((x,y),\cdot)\|+\int_0^T|p(x,t)-p(y,t)|dt\geq 1.
\end{aligned}
\end{align*}

For every $b\in M_{\text{fin}}(X^2)$, due to the Dini condition (see assumption (IV)) and the Jensen inequality, we get
\begin{align*}
\begin{aligned}
\|Q^1_{h_i}b\|=\int_{X^2}Q^1_{h_i}((x,y),X^2)b(dx\times dy)
&=\int_{X^2}\|Q^1_{h_i}((x,y),\cdot)\|b(dx\times dy)\\
&\geq\|b\|-\int_{X^2}\omega(\varrho(x,y))b(dx\times dy)\geq\|b\|-\omega(\phi(b)).
\end{aligned}
\end{align*}
Then, by $(\ref{Q1Qnb})$,
\begin{align*}
\begin{aligned}
\|Q^n_{h_1,\ldots,h_n}b\|=\int_{X^2}Q_{h_n}^1((x,y),\cdot)(Q^{n-1}_{h_1,\ldots,h_{n-1}}b)(dx\times dy)&\geq\|Q^{n-1}_{h_1,\ldots,h_{n-1}}b\|-\omega(\phi(Q^{n-1}_{h_1,\ldots,h_{n-1}}b))\\
&\geq \|b\|-\omega(\phi(Q^1_{h_1}b))-\ldots-\omega(\phi(Q^n_{h_1,\ldots,h_n}b)).
\end{aligned}
\end{align*}
Following $(\ref{prop:phiQb})$ and recalling that $\omega$ is non-decreasing, we obtain
\begin{align*}
\begin{aligned}
\|Q^n_{h_1,\ldots,h_n}b\|\geq\|b\|-\sum_{i=1}^{n}\omega(a^{i-1}\phi(b)).
\end{aligned}
\end{align*}

%It is assumed that $(x,y)\in K_{\varkappa}$, which means that $\varrho(x,y)\leq \varkappa^{-1}2c$. Hence, $\phi(b)\leq \varkappa^{-1}2c\|b\|$. 
See $(\ref{varphi})$ to recall the definition of $\varphi$. Thanks to assumption (IV), we know that $\lim_{t\to 0}\varphi(t) = 0$. Hence, we may choose $r>0$ such that if $\varrho(x,y)<r$ and therefore $a^{-1}\phi(b)\leq ra^{-1}\|b\|$, then $\sum_{i=1}^{n}\omega(a^{i-1}\phi(b))\leq\varphi(a^{-1}\phi(b))<\frac{1}{2}\|b\|$.

If $\text{supp }b\subset C_r$, then we obtain
\begin{align}\label{b}
\|Q_{h_1,h_2,\ldots}^{\infty}b\|\geq\frac{\|b\|}{2}.
\end{align}
Fix $\varkappa\in(0,1-a)$. It is clear that $K_{\varkappa}\subset C_{\varkappa^{-1} 2c}$. If we define 
$n_0:=\min\{n\in N:\: a^n(\varkappa)^{-1}2c<r\}$, 
then $C_{a^{n_0}\varkappa^{-1} 2c}\subset C_r$. 
Remembering that $Q_{h_1,\ldots,h_n,h_{n+1},\ldots,h_{n+m}}^{n+m}((x,y),\cdot)=(Q^m_{h_{n+1},\ldots,h_{n+m}}Q_{h_1,\ldots,h_n}^n)((x,y),\cdot)$ and using the Markov property, we obtain
\[Q_{h_1,h_2,\ldots}^{\infty}((x,y),X^2)= (Q^{\infty}_{h_{n_0+1},\ldots} Q^{n_0}_{h_1,\ldots,h_{n_0}})((x,y),X^2).\]
Then, according to $(\ref{b})$ and Lemma $\ref{lemma2}$, we obtain
\begin{align*}
\begin{aligned}
\|Q_{h_1,h_2,\ldots}^{\infty}((x,y),\cdot)\|
&=\|(Q^{\infty}_{h_{n_0+1},\ldots} Q^{n_0}_{h_1,\ldots,h_{n_0}})((x,y),\cdot)\|
\geq\frac{\|Q^{n_0}_{h_1,\ldots,h_{n_0}}((x,y),\cdot)|_{C_r}\|}{2}\\
&=\frac{Q^{n_0}_{h_1,\ldots,h_{n_0}}((x,y),C_r)}{2}
\geq\frac{Q^{n_0}_{h_1,\ldots,h_{n_0}}((x,y),C_{a^{n_0}\varkappa^{-1} 2c})}{2}
\geq\frac{\bar{\gamma}^{n_0}}{2}
\end{aligned}
\end{align*}
for $(x,y)\in K_{\varkappa}$. This finishes the proof.
\end{proof}

\begin{definition}
Coupling time $\tau:(X^2\times\{0,1\})^{\infty}\to N$ is defined as follows
\[\tau((x_n,y_n,\theta_n)_{n\in N})=\inf\{n\in N:\;\theta_k=1\;\text{ for }k\geq n\}.\] 
As a convention, we put $\tau((x_n,y_n,\theta_n)_{n\in N})=\infty$, if there is no $n\in N$ such that $\theta_k=1\;\text{ for every }k\geq n$.
\end{definition}

\begin{thm}\label{theorem3}
There exist $\tilde{q}\in(0,1)$ and $C_3>0$ such that
\[E_{x,y}[\tilde{q}^{-\tau}]\leq C_3(1+\bar{V}(x,y))\quad \text{for $(x,y)\in X^2$.}\]
\end{thm}

\begin{proof}[Proof]
Fix $\varkappa\in(0,1-a)$ and $(x,y)\in X^2$. To simplify notation, we write $\alpha=(a+\varkappa)^{-\frac{1}{2}}$. 
Let $d$ be the random moment of the first visit in $K_{\varkappa}$. 
Suppose that
\[d_1=d,\quad d_{n+1}=d_n+d\circ \Gamma_{d_n},\]
where $n\in N$ and $\Gamma_n$ are shift operators on $(X^2\times\{0,1\})^{\infty}$, i.e. 
$\Gamma_n((x_k,y_k,\theta_k)_{k\in N})=(x_{k+n},y_{k+n},\theta_{k+n})_{k\in N}$.
Theorem $\ref{theorem1}$ implies that every $d_n$ is $C_{h_1,h_2,\ldots}^{\infty}((x,y),\cdot)$-a.s. finished. The strong Markov property shows that
\[E_{x,y}[\alpha^d\circ \Gamma_{d_n}|F_{d_n}]=E_{(x_{d_n},y_{d_n})}[\alpha^d]\quad \text{for }n\in N,\]
where $F_{d_n}$ denotes the $\sigma$-algebra on $(X^2\times\{0,1\})$ generated by $d_n$ and $\Psi=(x_n,y_n)_{n\in N}$ is the non-homogenous Markov chain with sequence of transition probability functions $(\{C_{h_i}^{1}((x,y),\cdot):x,y\in X\})_{i\in N}$. 
By Theorem $\ref{theorem1}$ and the definition of $K_{\varkappa}$, we obtain
\[E_{x,y}[\alpha^{d_{n+1}}]=E_{x,y}\Big[\alpha^{d_n}E_{(x_{d_n},y_{d_n})}[\alpha^d]\Big]\leq E_{x,y}[\alpha^{d_n}](C_1\varkappa^{-1} 2c+C_2).\]
Fix $\eta=C_1\varkappa^{-1} 2c+C_2$. Consequently,
\begin{align}\label{condition1}
E_{x,y}[\alpha^{d_{n+1}}]\leq\eta^n E_{x,y}[\alpha^d]\leq\eta^n[C_1\bar{V}(x,y)+C_2].
\end{align}
We define $\hat{\tau}((x_n,y_n,\theta_n)_{n\in N})=\inf\{n\in N:\; (x_n,y_n)\in K_{\varkappa}, \ \;\theta_k=1\:\text{ for }k\geq n\}$ 
and $\sigma=\inf\{n\in N:\;\hat{\tau}=d_n\}$. 
By Theorem $\ref{theorem2}$, there is $n_0\in N$ such that
\begin{align}\label{condition2}
\hat{C}_{h_1,h_2,\ldots}^{\infty}((x,y),\{\sigma>n\})\leq(1-\frac{\bar{\gamma}^{n_0}}{2})^n \quad\text{for }n\in N.
\end{align}
Let $p>1$. By the H\"{o}lder inequality, $(\ref{condition1})$ and $(\ref{condition2})$, we obtain
\begin{align*}
\begin{aligned}
E_{x,y}[\alpha^{\frac{\hat{\tau}}{p}}]&\leq\sum_{k=1}^{\infty}E_{x,y}[\alpha^{\frac{d_k}{p}}1_{\sigma=k}]
\leq\sum_{k=1}^{\infty}\Big(E_{x,y}[\alpha^{d_k}]\Big)^{\frac{1}{p}}\Big(\hat{C}_{h_1,h_2,\ldots}^{\infty}((x,y),\{\sigma=k\})\Big)^{(1-\frac{1}{p})}\\
&\leq[C_1\bar{V}(x,y)+C_2]^{\frac{1}{p}}\eta^{-\frac{1}{p}}\sum_{k=1}^{\infty}\eta^{\frac{k}{p}}(1-\frac{1}{2}\bar{\gamma}^{n_0})^{(k-1)(1-\frac{1}{p})}\\
&=[C_1\bar{V}(x,y)+C_2]^{\frac{1}{p}}\eta^{-\frac{1}{p}}(1-\frac{1}{2}\bar{\gamma}^{n_0})^{-(1-\frac{1}{p})}\sum_{k=1}^{\infty}\Big[\Big(\frac{\eta}{1-\frac{1}{2}\bar{\gamma}^{n_0}}\Big)^{\frac{1}{p}}(1-\frac{1}{2}\bar{\gamma}^{n_0})\Big]^k.
\end{aligned}
\end{align*}
For $p$ sufficiently large and $\tilde{q}=\alpha^{-\frac{1}{p}}$, we get
\[E_{x,y}[\tilde{q}^{-\hat{\tau}}]=E_{x,y}[\alpha^{\frac{\hat{\tau}}{p}}]\leq(1+\bar{V}(x,y))C_3\]
for some $C_3$. Since $\tau\leq\hat{\tau}$, we finish the proof.
\end{proof}

\begin{lemma}\label{theorem_new}
Let $f\in\mathcal{L}$. Then, there exist $q\in(0,1)$ and $C_5>0$ such that
\begin{align*}
\begin{aligned}
\int_{X^2}|f(u)-f(v)|(\Pi_{X^2}^*\Pi^*_n\hat{C}^{\infty}_{h_1,h_2\ldots}((x,y),\cdot))(du\times dv)\leq q^nC_5(1+\bar{V}(x,y))\;\text{for every $x,y\in X$, $n\in N$,}
\end{aligned}
\end{align*}
where $\Pi^*_n:(X^2\times\{0,1\})^{\infty}\to X^2\times\{0,1\}$ are the projections on the $n$-th component and $\Pi^*_{X^2}:X^2\times\{0,1\}\to X^2$ is the projection on $X^2$.
\end{lemma}

\begin{proof}[Proof]
For $n\in N$ we define sets
\begin{align*}
\begin{aligned}
A_{\frac{n}{2}}=\{t\in(X^2\times\{0,1\})^{\infty}:\:\tau(t)\leq\frac{n}{2}\},\\
B_{\frac{n}{2}}=\{t\in(X^2\times\{0,1\})^{\infty}:\:\tau(t)>\frac{n}{2}\}.
\end{aligned}
\end{align*}
Note that $A_{\frac{n}{2}}\cap B_{\frac{n}{2}}=\emptyset$ and $A_{\frac{n}{2}}\cup B_{\frac{n}{2}}=(X^2\times\{0,1\})^{\infty}$, so, for $n\in N$, we have
\[\hat{C}_{h_1,h_2,\ldots}^{\infty}((x,y),\cdot)=\hat{C}_{h_1,h_2,\ldots}^{\infty}((x,y),\cdot)|_{A_{\frac{n}{2}}}+\hat{C}_{h_1,h_2,\ldots}^{\infty}((x,y),\cdot)|_{B_{\frac{n}{2}}}.\]
Hence,
\begin{align*}
\begin{aligned}
&\int_{X^2}|f(u)-f(v)|(\Pi_{X^2}^*\Pi_n^*\hat{C}_{h_1,h_2,\ldots}^{\infty}((x,y),\cdot)|_{A_{\frac{n}{2}}})(du\times dv)\\
&\qquad +\int_{X^2}|f(u)-f(v)|(\Pi_{X^2}^*\Pi_n^*\hat{C}_{h_1,h_2,\ldots}^{\infty}((x,y),\cdot)|_{B_{\frac{n}{2}}})(du\times dv)\\
&\leq\int_{X^2}\varrho(u,v)(\Pi_{X^2}^*\Pi_n^*\hat{C}_{h_1,h_2,\ldots}^{\infty}((x,y),\cdot)|_{A_{\frac{n}{2}}})(du\times dv)+2\hat{C}_{h_1,h_2,\ldots}^{\infty}((x,y),B_{\frac{n}{2}}).
\end{aligned}
\end{align*}
Note that, by iterative application of $(\ref{prop:phiQb})$, we obtain
\begin{align*}
\begin{aligned}
\int_{X^2}\varrho(u,v)(\Pi_{X^2}^*\Pi_n^*\hat{C}_{h_1,h_2,\ldots}^{\infty}((x,y),\cdot)|_{A_{\frac{n}{2}}})(du,dv)
&=\phi(\Pi_{X^2}^*\Pi_n^*(\hat{C}_{h_1,h_2,\ldots}^{\infty}((x,y),\cdot)|_{A_{\frac{n}{2}}}))\\
&\leq a^{\lfloor\frac{n}{2}\rfloor}\phi(\Pi_{X^2}^*\Pi_{\lfloor\frac{n+1}{2}\rfloor}^*(\hat{C}_{h_1,h_2,\ldots}^{\infty}((x,y),\cdot)|_{A_{\frac{n}{2}}})).
\end{aligned}
\end{align*}
Then, it follows from $(\ref{prop:barV})$ and $(\ref{prop:phi})$ that
\begin{align*}
\begin{aligned}
\phi(\Pi_{X^2}^*\Pi_{\lfloor\frac{n+1}{2}\rfloor}^*(\hat{C}_{h_1,h_2,\ldots}^{\infty}((x,y),\cdot)|_{A_{\frac{n}{2}}}))\leq a^{\lfloor\frac{n+1}{2}\rfloor}\bar{V}(x,y)+\frac{2c}{1-a}.
\end{aligned}
\end{align*}
We obtain
\begin{align*}
&\int_{X^2}|f(u)-f(v)|(\Pi_{X^2}^*\Pi^*_n\hat{C}^{\infty}_{h_1,h_2\ldots}((x,y),\cdot))(du\times dv)\\
&\qquad\leq a^{\lfloor\frac{n}{2}\rfloor}\Big[a^{\lfloor\frac{n+1}{2}\rfloor}\bar{V}(x,y)+\frac{2c}{1-a}\Big]+2\hat{C}_{h_1,h_2,\ldots}^{\infty}((x,y),B_{\frac{n}{2}}).
\end{align*}
It follows from Theorem $\ref{theorem3}$ and the Chebyshev inequality that
\begin{align*}
\begin{aligned}
\hat{C}_{h_1,h_2,\ldots}^{\infty}((x,y),B_{\frac{n}{2}})=\hat{C}_{h_1,h_2,\ldots}^{\infty}((x,y),\{\tau>\frac{n}{2}\})
&=\hat{C}_{h_1,h_2,\ldots}^{\infty}((x,y),\{\tilde{q}^{-\tau}\geq\tilde{q}^{-\frac{n}{2}}\})\\
&\leq\frac{E_{x,y}[\tilde{q}^{-\tau}]}{\tilde{q}^{-\frac{n}{2}}}\leq\tilde{q}^{\frac{n}{2}}C_3(1+\bar{V}(x,y)),
\end{aligned}
\end{align*}
for some $\tilde{q}\in(0,1)$ and $C_3>0$. 
Finally,
\[\int_{X^2}|f(u)-f(v)|(\Pi_{X^2}^*\Pi^*_n\hat{C}^{\infty}_{h_1,h_2\ldots}((x,y),\cdot))(du\times dv)\leq a^{\lfloor\frac{n}{2}\rfloor}C_4(1+\bar{V}(x,y))+2\tilde{q}^{\frac{n}{2}}C_3(1+\bar{V}(x,y)),\]
where $C_4=\max\{a^{\frac{1}{2}},(1-a)^{-1}2c\}$. Setting $q:=\max\{a^{\frac{1}{2}},\tilde{q}^{\frac{1}{2}}\}$ and $C_5:=C_4+2C_3$, gives our claim.
\end{proof}

\begin{remark}\label{rem_g}
If $g:X\to R$ is an arbitrary bounded and Lipschitz function with constant $L_g$, then, there are $q\in(0,1)$ and $C_5>0$, exactly the same as in Lemma \ref{theorem_new}, for which we obtain
\begin{align*}
\begin{aligned}
\int_{X^2}|g(u)-g(v)|(\Pi_{X^2}^*\Pi^*_n\hat{C}^{\infty}_{h_1,h_2\ldots}((x,y),\cdot))(du\times dv)\leq Gq^nC_5(1+\bar{V}(x,y))\;\text{for every $x,y\in X$, $n\in N$,}
\end{aligned}
\end{align*}
where $G:=\max\{L_g,\sup_{x\in X}|g(x)|\}$.
\end{remark}

\begin{thm}\label{theorem4}
There exist $q\in(0,1)$ and $C_5>0$ such that
\[\|\Pi^n_{h_1,\ldots,h_n}(x,\cdot)-\Pi^n_{h_1,\ldots,h_n}(y,\cdot)\|_{\mathcal{L}}\leq q^nC_5(1+\bar{V}(x,y))\quad\text{for }\:x,y\in X\;\text{ and }\:n\in N.\]
\end{thm}

\begin{proof}[Proof]
The theorem is a consequence of Lemma \ref{theorem_new}. It is enough to observe that 
\begin{align*}
\begin{aligned}
\|\Pi^n_{h_1,\ldots,h_n}(x,\cdot)-\Pi^n_{h_1,\ldots,h_n}(y,\cdot)\|_{\mathcal{L}}&=\sup_{f\in\mathcal{L}}\Big|\int_{X}f(z)(\Pi^n_{h_1,\ldots,h_n}(x,\cdot)-\Pi^n_{h_1,\ldots,h_n}(y,\cdot))(dz)\Big|\\
&=\sup_{f\in\mathcal{L}}\Big|\int_{X^2}(f(z_1)-f(z_2))(\Pi_{X^2}^*\Pi_n^*\hat{C}_{h_1,h_2,\ldots}^{\infty}((x,y),\cdot))(dz_1\times dz_2)\Big|.
\end{aligned}
\end{align*}
Hence, using the argument of Lemma \ref{theorem_new}, we obtain 
\[\|\Pi^n_{h_1,\ldots,h_n}(x,\cdot)-\Pi^n_{h_1,\ldots,h_n}(y,\cdot)\|_{\mathcal{L}}\leq q^nC_5(1+\bar{V}(x,y)),\]
which finishes the proof.
\end{proof}

%Note that we do not have any invariant measure in the auxiliary model. However, if we return to the initial model, we do obtain an interesting result.

\section{Exponential rate of convergence - proof of Theorem $1$}

Note that we may write
\[P^n_{\varepsilon}\mu(\cdot)=\int_X\int_{B(0,\varepsilon)}\ldots\int_{B(0,\varepsilon)}\Pi^n_{h_1,\ldots,h_n}(x,\cdot)\nu^{\varepsilon}(dh_1)\ldots\nu^{\varepsilon}(dh_n)\mu(dx).\]

Comparing this approach with Remark 1 and Lemma 1, we see that $P_{\varepsilon}$ is Feller and, for every $n\in N$, it satisfies the following property 
\begin{align}\label{prop:<V,P^n_epsilon_mu>}
\langle V,P^n_{\varepsilon}\mu\rangle \leq a^n\langle V,\mu\rangle +\frac{c}{1-a}.
\end{align}
%Now, fix $x\in X$. Note that there exists $M_x:=aV(x)+\frac{c}{1-a}<\infty$ such that $U^nV(x)\leq M_x$, for every $n\in N$.

%The remark is that these observations do not lead us immediately to any conclusions about asymptotic stability of the model and so the reasoning based on some coupling techniques remains crucial to solve the problem.

We present the proof of Theorem $1$ below.

\begin{proof}[Proof of Theorem 1]

Let $\mu_1,\mu_2\in M_{1}^1(X)$. We first want to evaluate $\|P^n_{\varepsilon}\mu_1-P^n_{\varepsilon}\mu_2\|_{\mathcal{L}}$. Let $f\in\mathcal{L}$. We obtain

\begin{align*}
\begin{aligned}
&|\langle f,P^n_{\varepsilon}\mu_1-P^n_{\varepsilon}\mu_2\rangle |\\&=\Big|\int_X\int_{B(0,\varepsilon)}\ldots\int_{B(0,\varepsilon)}\int_Xf(z)\Pi^n_{h_1,\ldots,h_n}(x,dz)\nu^{\varepsilon}(dh_1)\ldots\nu^{\varepsilon}(dh_n)\mu_1(dx)\\
&\quad 
-\int_X\int_{B(0,\varepsilon)}\ldots\int_{B(0,\varepsilon)}\int_Xf(z)\Pi^n_{h_1,\ldots,h_n}(y,dz)\nu^{\varepsilon}(dh_1)\ldots\nu^{\varepsilon}(dh_n)\mu_2(dy)\Big|\\
&=\Big|\int_X\Big[\int_X\int_{B(0,\varepsilon)}\ldots\int_{B(0,\varepsilon)}\int_Xf(z)\Pi^n_{h_1,\ldots,h_n}(x,dz)\nu^{\varepsilon}(dh_1)\ldots\nu^{\varepsilon}(dh_n)\mu_1(dx)\Big]\mu_2(dy)\\  &\quad -\int_X\Big[\int_X\int_{B(0,\varepsilon)}\ldots\int_{B(0,\varepsilon)}\int_Xf(z)\Pi^n_{h_1,\ldots,h_n}(y,dz)\nu^{\varepsilon}(dh_1)\ldots\nu^{\varepsilon}(dh_n)\mu_2(dy)\Big]\mu_1(dx)\Big|\\
\end{aligned}
\end{align*}

Now, following the result from Theorem \ref{theorem4},
\[\|\Pi^n_{h_1,\ldots,h_n}(x,\cdot)-\Pi^n_{h_1,\ldots,h_n}(y,\cdot)\|_{\mathcal{L}}\leq q^nC_5(1+\bar{V}(x,y)),\]
where $q$ and $C_5$ are independent of choice of $h_1,h_2,\ldots$, we obtain the inequality
\begin{align*}
\begin{aligned}
&|\langle f,P^n_{\varepsilon}\mu_1-P^n_{\varepsilon}\mu_2\rangle |\\
&=\int_X\int_X\Big[\int_{B(0,\varepsilon)}\ldots\int_{B(0,\varepsilon)}\Big|\int_Xf(z)\Pi^n_{h_1,\ldots,h_n}(x,dz)-\int_Xf(z)\Pi^n_{h_1,\ldots,h_n}(y,dz)\Big|\\
&\qquad\nu^{\varepsilon}(dh_1)\ldots\nu^{\varepsilon}(dh_n)\Big]\mu_1(dx)\mu_2(dy)\\
&\leq\int_X\int_X\Big[\int_{B(0,\varepsilon)}\ldots\int_{B(0,\varepsilon)}\|\Pi^n_{h_1,\ldots,h_n}(x,\cdot)-\Pi^n_{h_1,\ldots,h_n}(y,\cdot)\|_{\mathcal{L}}\nu^{\varepsilon}(dh_1)\ldots\nu^{\varepsilon}(dh_n)\Big]\mu_1(dx)\mu_2(dy)\\
&\leq q^nC_5\int_X\int_X\Big[\int_{B(0,\varepsilon)}\ldots\int_{B(0,\varepsilon)}(1+\bar{V}(x,y))\nu^{\varepsilon}(dh_1)\ldots\nu^{\varepsilon}(dh_n)\Big]\mu_1(dx)\mu_2(dy)\\
&=q^nC_5\int_X\int_X(1+\bar{V}(x,y))\mu_1(dx)\mu_2(dy),\\
\end{aligned}
\end{align*}
where measures $\mu_1,\mu_2\in M^1_{1}(X)$.

Now, set $\mu_1:=\mu\in M_1^1(X)$ and $\mu_2:=P^m_{\varepsilon}\mu\in M_1^1(X)$, for arbitrary $m\in N$. Note that it follows form Lemma 1 that $P^m_{\varepsilon}\mu$ is with finite first moment if $\mu\in M_1^1$. We obtain 
\begin{align*}
\begin{aligned}
\|P^n_{\varepsilon}\mu-P^{n+m}_{\varepsilon}\mu\|_{\mathcal{L}}&\leq q^nC_5\int_X\int_X(1+\bar{V}(x,y))\mu(dx)P^m_{\varepsilon}\mu(dy)\\
&=q^nC_5\Big[1+\int_X V(x)\mu(dx)+\int_X V(y)P^m_{\varepsilon}\mu(dy)\Big]\leq q^nC_6
%&\leq q^n C_6\int_XV(y)P^m_{\varepsilon}\mu(dy)\\
%&=q^nC_6\int_XU^m_{\varepsilon}V(y)\mu(dy)\\
%&\leq q^nC_6\int_X[a^mV(y)+c]\mu(dy)
\end{aligned}
\end{align*}
for some constant $C_6$. Hence, $(P^n_{\varepsilon}\mu)_{n\in N}$ is a Cauchy sequence in $(M_1(X),\|\cdot\|_{\mathcal{L}})$. 
It is then proven, because of completeness of the space, that $(P^n_{\varepsilon}\mu)_{n\in N}$ converges in $(M_1(X), \|\cdot\|_{\mathcal{L}})$. Put $\mu_*(\mu)=\lim_{n\to\infty}P^n_{\varepsilon}\mu$. As mentioned before, we know that $P_{\varepsilon}$ is a Feller operator and this impies that the measure $\mu_*(\mu)$ is invariant. Then, for $\mu_1,\mu_2\in M_1^1(X)$ and every $\epsilon>0$, we have
\[\|\mu_*(\mu_1)-\mu_*(\mu_2)\|_{\mathcal{L}}\leq \|\mu_*(\mu_1)-P^n_{\varepsilon}\mu_1\|_{\mathcal{L}}+\|P^n_{\varepsilon}\mu_1-P^n_{\varepsilon}\mu_2\|_{\mathcal{L}}+\|\mu_*(\mu_2)-P^n_{\varepsilon}\mu_2\|_{\mathcal{L}}<\epsilon\]
for sufficiently large $n\in N$. Hence, we have the invariant measure $\mu_*:=\mu_*(\mu)$ which is unique in $M_1^1(X)$. 
We should make it clear that $\mu_*\in M_1^1(X)$. Note that we can take a non-decreasing sequence $(V_k)_{k\in N}$ such that $V_k(y)=\min\{k,V(y)\}$, for every $k\in N$ and $y\in X$. Fix $x\in X$. From the first part of the proof, we know that $\langle f,P^n_{\varepsilon}\delta_x\rangle $ converges to $\langle f,\mu_*\rangle $ for every $f\in \mathcal{L}$, which means, by the Aleksandrov theorem (see Theorem 11.3.3 in \cite{dudley}), that $P^n_{\varepsilon}\delta_x$ converges weakly to $\mu_*$, since both measures are probabilistic. Hence, for all $k\in N$, $V_k\in C(X)$ and we obtain
\[\lim_{n\to\infty}\int_XV_k(y)P^n_{\varepsilon}\delta_x(dy)=\int_XV_k(y)\mu_*(dy).\]
Note that, according to (\ref{prop:<V,P^n_epsilon_mu>}), for every $n\in N$, \[\langle V_k,P^n_{\varepsilon}\delta_x\rangle =a^n\langle V_k,\delta_x\rangle +(1-a)^{-1}c \leq a^nV_k(x)+(1-a)^{-1}c\] %so the sequence $(\langle V_k,P^n_{\varepsilon}\delta_x\rangle )_{k\in N}$ is uniformly bounded 
and, additionally, 
\[\langle V_k,\mu_*\rangle =\lim_{n\to\infty}\langle V_k,P^n_{\varepsilon}\delta_x\rangle \leq (1-a)^{-1}c,\]
so the sequence $(\langle V_k,\mu_*\rangle )_{k\in N}$ is bounded. 
Because $(V_k)_{k\in N}$ is non-negative and non-decreasing, we may use the Monotone Convergence Theorem to obtain
\[\int_XV(y)\mu_*(dy)=\lim_{k\to\infty}\int_XV_k(y)\mu_*(dy).\]%\leq (1-a)^{-1}c
Then, $V$ is integrable with respect to $\mu_*$, so $\mu_*$ is with finite first moment.

Keeping in mind that $\bar{V}(x,y)=V(x)+V(y)$, the exponential rate of convergence to the unique invariant measure $\mu_*\in M_1^1(X)$ derives from the following estimates
\begin{align*}
\begin{aligned}
\|P^n_{\varepsilon}{\mu}-\mu_*\|_{\mathcal{L}}
\leq\int_X\int_X q^nC_5(1+\bar{V}(x,y))\mu_*(dy)\mu(dx)
\leq q^nC,
\end{aligned}
\end{align*}
where $C:=\int_X\int_XC_5(1+\bar{V}(x,y))\mu_*(dy)\mu(dx)<\infty$ for $\mu\in M^1_1(X)$. Finally, since $C$ is dependant only on $\mu$, the proof is complete.

\end{proof}

\section{Central Limit Theorem - proof of Theorem $2$}

Let us first make the follownig observation. 
\begin{lemma}\label{second_moment}
If $\mu\in M_1^2(X)$ is with finite second moment, then $\langle V^2,P^n_{\varepsilon}\mu\rangle \,<\infty$ and therefore $\mu_*:=\mu_*(\mu)=\lim_{n\to\infty}P^n_{\varepsilon}\mu$ has finite second moment.
\end{lemma}
\begin{proof}[Proof]
Let $\mu\in M_1^2(X)$. Fix $x\in X$, $n\geq 1$. Recall that $\Lambda<1/2$ and $c$ are given by (\ref{def:b}) and (\ref{def:c}), respectively. Moreover, $\tilde{x}^x_n$ is of the form (\ref{def:tilde_x}). Reasoning as in Lemma 1, we obtain

\begin{align*}
\begin{aligned}
&\langle V^2,P^n_{h_1,\ldots,h_n}\mu\rangle \\&
=\int_X\Big[\int_0^T\ldots\int_0^T\varrho(\tilde{x}^x_n,\bar{x})p(\tilde{x}^x_{n-1},t_{n})\ldots p(x,t_1)dt_n\ldots dt_1\Big]\mu(dx)\\
&\leq\int_X\int_0^T\ldots\int_0^T\Big[2\varrho^2(\tilde{x}^x_n,\tilde{x}^{\bar{x}}_n)+2\varrho^2(\tilde{x}^{\bar{x}}_n,\bar{x})\Big]p(\tilde{x}^x_{n-1},t_{n})\ldots p(x,t_1)dt_n\ldots dt_1\mu(dx)\\
&\leq 2\int_X\int_0^T\ldots\int_0^T\Big[\lambda^2(\tilde{x}^x_n,t_n)\ldots\lambda^2(x,t_1)\varrho^2(x,\bar{x})+2\varrho^2(\bar{x},\tilde{x}^{\bar{x}}_1)+4\varrho^2(\tilde{x}^{\bar{x}}_1,\tilde{x}^{\bar{x}}_2)+4\varrho^2(\tilde{x}^{\bar{x}}_2,\tilde{x}^{\bar{x}}_n)\Big]\\
&\qquad p(\tilde{x}^x_{n-1},t_{n})\ldots p(x,t_1)dt_n\ldots dt_1\mu(dx)\\
&\leq 2\Lambda^n\langle V^2,\mu\rangle +2^2c^2(1+2\Lambda+\ldots+2^n\Lambda^n)\\
&\leq 2\Lambda^n\langle V^2,\mu\rangle +4\frac{c^2}{1-2\Lambda},
\end{aligned}
\end{align*}
Estimates are independent of choice of sequence $(h_n)_{n\in N}$ and therefore $\langle V^2,P^n_{\varepsilon}\mu\rangle \,<\;2\Lambda^n\langle V^2,\mu\rangle +4c^2(1-2\Lambda)^{-1}$. We take a non-decreasing sequence $(V^2_k)_{k\in N}$ such that $V^2_k(y)=\min\{k,V^2(y)\}$, for every $k\in N$ and $y\in X$. We know that $P^n_{\varepsilon}\mu$ converges weakly to $\mu_*$. Hence, for all $k\in N$, $V^2_k\in C(X)$ and
\[\lim_{n\to\infty}\langle V^2_k,P^n_{\varepsilon}\mu\rangle =\langle V^2_k,\mu_*\rangle \,<\;4c^2(1-2\Lambda)^{-1},\]
so the sequence $(\langle V^2_k,\mu_*\rangle )_{k\in N}$ is bounded. 
Because $(V^2_k)_{k\in N}$ is non-negative and non-decreasing, we may use the Monotone Convergence Theorem to obtain
\[\langle V^2,\mu_*\rangle =\lim_{k\to\infty}\langle V^2_k,\mu_*\rangle \,\]
so, indeed, $\mu_*$ is with finite second moment.
\end{proof}

Let $\eta_n^{\mu}$ and $\Phi\eta_n^\mu$ be as in Section $\ref{sec:idea}$. In particular, $\eta_n^*$ and $\eta_n^x$ are defined for the Markov chains with the same transition probability function $\Pi_{\varepsilon}$ and initial distributions $\mu_*$ and $\delta_x$, respectively. Further, let $g:X\to R$ be a bounded and Lipschitz continuous function, with constant $L_g$, which satisfies $\langle g,\mu_*\rangle =0$.

Central Limit Theorems for ergodic stationary Markov chains have already been proven in many papers. See, for example, Theorem $1$ and the subsequent Corollary $1$ in \cite{woodr} by Maxwell and Woodroofe. %Note that our model satisfies the assumptions of that paper thanks to the choice of $g$ and, already proven in Theorem 1, exponential rate of~convergence.%, what is established in the following lemma. 
The following lemma implies that assumptions of Theorem 1 and Corollary 1 (\cite{woodr}) are satisfied.

\begin{lemma}%\label{condition_woodroofe}
Let $g:X\to R$ be a bounded and Lipschitz continuous function with constant $L_g$. Additionally, $\langle g,\mu_*\rangle =0$. Then,
\begin{align}\label{condition_woodroofe}
\begin{aligned}
\sum_{n=1}^{\infty}n^{-3/2}\Big[\int_X\Big(\sum_{k=0}^{n-1}\langle g,P^{k}_{\varepsilon}\delta_x\rangle \Big)^2\mu_*(dx)\Big]^{1/2}<\infty
\end{aligned}
\end{align}
\end{lemma}
\begin{proof}[Proof]
Note that, by Lemma \ref{theorem_new} and Remark \ref{rem_g}, 
\begin{align*}
\begin{aligned}
\sum_{k=0}^{n-1}\langle g,P^{k}_{\varepsilon}\delta_x\rangle &=\sum_{k=0}^{n-1}\Big(\langle g,P^{k}_{\varepsilon}\delta_x\rangle -\langle g,\mu_*\rangle \Big)\\
&=\sum_{k=0}^{n-1}\int_X\Big[\int_Xg(z)(\Pi^k_{\varepsilon}(x,\cdot)-\Pi^k_{\varepsilon}(y,\cdot))(dz)\Big]\mu_*(dy)\\
&=\sum_{k=0}^{n-1}\int_X\Big[\int_{X^2}(g(z_1)-g(z_2))(\Pi^*_{X^2}\Pi^*_k\hat{C}^{\infty}_{h_1,h_2,\ldots}((x,y),\cdot))(dz_1\times dz_2)\Big]\mu_*(dy)\\
&\leq \sum_{k=0}^{n-1}Gq^nC_5\int_{X^2}(1+\bar{V}(x,y))\mu_*(dy).
\end{aligned}
\end{align*}
Then, for every $x\in X$, $n\in N$,
\[\sum_{k=0}^{n-1}\langle g,P^{k}_{\varepsilon}\delta_x\rangle \leq GC_5\frac{1-q^n}{1-q}\int_{X^2}(1+\bar{V}(x,y))\mu_*(dy)\leq C_9(1+V(x)),\]
where $C_9:=GC_5(1-q)^{-1}(1+\int_XV(y)\mu_*(dy))$. Keeping in mind that $\mu_*$ is with finite second moment, we obtain that $(\ref{condition_woodroofe})$ is not bigger than
\begin{align*}
\begin{aligned}
\sum_{n=1}^{\infty}n^{-3/2}[C_9^2\langle 1+2V+V^2,\mu_*\rangle ]^{1/2}<\infty 
\end{aligned}
\end{align*} 
and the proof is complete.
\end{proof}

Hence, by applying Corollary $1$, we obtain that $\Phi\eta_n^*$ converges to the normal distribution in Levy metric, as $n\to\infty$, which equivalently means that the distributions converge weakly to each other (see \cite{kurtz} for proofs).

Now, the idea of the proof is based on the following remark.

\begin{remark}\label{enough}
Note that, to complete the proof of Theorem $\ref{CTG}$, it is enough to establish that $\Phi{\eta_n^{\mu}}$ converges weakly to $\Phi{\eta_n^*}$, as $n\to\infty$. Equivalently, it is enough to show that $\lim_{n\to\infty}\|\Phi{\eta_n^{\mu}}-\Phi{\eta_n^*}\|_{\mathcal{L}}=0$, since weak convergence is metrised by the Fourtet-Mourier norm.
\end{remark}

\begin{proof}[Proof of Theorem 2]
Set $x,y\in X$ and choose arbitrary $f\in\mathcal{L}$. Suppose that we know that the following convergence is satisfied, as $n\to\infty$,
\begin{align}\label{prop:xy}
\begin{aligned}
\Big|\int_Rf(u)\Phi\eta_n^x(du)-\int_Rf(v)\Phi\eta_n^y(dv)\Big|\to 0.
\end{aligned}
\end{align}
Then, by the Dominated Convergence Theorem, we obtain
\begin{align}\label{proven}
\begin{aligned}
&\Big|\int_Rf(u)\Phi\eta_n^{\mu}(du)-\int_Rf(v)\Phi\eta_n^*(dv)\Big|\\
&\qquad\leq\int_X\int_X\Big|\int_Rf(u)\Phi\eta_n^x(du)-\int_Rf(v)\Phi\eta_n^y(dv)\Big|\mu(dx)\mu_*(dy)
\to 0,
\end{aligned}
\end{align}
as $n\to\infty$. Note that, by Theorem 11.3.3 in \cite{dudley}, $(\ref{proven})$ implies that $\Phi{\eta_n^{\mu}}$ converges weakly to $\Phi{\eta_n^*}$, as $n\to\infty$, which, according to Remark $\ref{enough}$, completes the proof of the CLT in the model. 
Now, it remains to show $(\ref{prop:xy})$. Note that
\begin{align}\label{prop:pre-estimate}
\begin{aligned}
\Big|\int_Rf(u)\Phi\eta_n^x(du)-\int_Rf(v)\Phi\eta_n^y(dv)\Big|
&=\Big|\int_{X^n}f\Big(\frac{g(u_1)+\ldots+g(u_n)}{\sqrt{n}}\Big)\Pi_{\varepsilon}^{1,\ldots,n}(x,du_1\times\ldots\times du_n)\\
&\qquad - \int_{X^n}f\Big(\frac{g(v_1)+\ldots+g(v_n)}{\sqrt{n}}\Big)\Pi_{\varepsilon}^{1,\ldots,n}(y,dv_1\times\ldots\times dv_n)  \Big|,
\end{aligned}
\end{align}
where $\Pi_{\varepsilon}^{1,\ldots,n}(x,\cdot)=
\int_{B(0,\varepsilon)}\ldots\int_{B(0,\varepsilon)}\Pi_{h_1,\ldots,h_n}^{1,\ldots,n}(x,\cdot)\nu^{\varepsilon}(dh_1)\ldots\nu^{\varepsilon}(dh_n)$ is a measure on $X^n$. We may write
\begin{align}\label{estimate}
\begin{aligned}
&\Big|\int_{X^n}\int_{X^n}\Big[f\Big(\frac{g(u_1)+\ldots+g(u_n)}{\sqrt{n}}\Big) - f\Big(\frac{g(v_1)+\ldots+g(v_n)}{\sqrt{n}}\Big)\Big]\\
&\qquad\qquad\qquad\Pi_{h_1,\ldots,h_n}^{1,\ldots,n}(x,du_1\times\ldots\times du_n)\Pi_{h_1,\ldots,h_n}^{1,\ldots,n}(y,dv_1\times\ldots\times dv_n)  \Big|\\
&\leq\int_{(X^2)^n}\Big| f\Big(\frac{g(u_1)+\ldots+g(u_n)}{\sqrt{n}}\Big) - f\Big(\frac{g(v_1)+\ldots+g(v_n)}{\sqrt{n}}\Big) \Big|\\ &\qquad\qquad\qquad\Big(\Pi^*_{X^{2n}}\Pi^*_{1,\ldots,n}\hat{C}^{\infty}_{h_1,h_2\ldots}((x,y),\cdot)\Big)(du_1\times\ldots\times du_n\times dv_1\times\ldots\times dv_n),
\end{aligned}
\end{align}
where $\Pi^*_{1,\ldots,n}:(X^2\times\{0,1\})^{\infty}\to(X^2\times\{0,1\})^n$ are the projections on the first $n$ components and $\Pi^*_{X^{2n}}:(X^2\times\{0,1\})^n\to X^{2n}$ is the projection on $X^{2n}$. 
Since $f$ is Lipschitz with constant $L_f$, we may further estimate $(\ref{estimate})$ from above
\begin{align*}
\begin{aligned}
&\frac{L_f}{\sqrt{n}}\int_{X^{2n}}\Big[|g(u_1)-g(v_1)|+\ldots+|g(u_n)-g(v_n)|\Big]
(\Pi^*_{X^{2n}}\Pi^*_{1,\ldots,n}\hat{C}^{\infty}_{h_1,h_2\ldots}((x,y),\cdot)\Big)((du_i\times dv_i)_{i=1}^n)\\
&= \frac{L_f}{\sqrt{n}}\sum_{i=1}^n\int_{X^2}|g(u_i)-g(v_i)|
(\Pi^*_{X^{2}}\Pi^*_{i}\hat{C}^{\infty}_{h_1,h_2\ldots}((x,y),\cdot)\Big)(du_i\times dv_i).
\end{aligned}
\end{align*}
Now, for every $1\leq i\leq n$, we refer to Lemma \ref{theorem_new} and Remark \ref{rem_g} to obseve that  $(\ref{estimate})$ is not bigger than
\[\frac{L_f G}{\sqrt{n}}\sum_{i=1}^nq^iC_5(1+\bar{V}(x,y))=n^{-\frac{1}{2}}L_fGC_5q\frac{1-q^n}{1-q}(1+\bar{V}(x,y)).\]
Note that the expression above is independent of choice of sequence $(h_n)_{n\in N}$ and, thanks to this, is also the upper bound of $(\ref{prop:pre-estimate})$. We go with $n$ to infinity and obtain $(\ref{prop:xy})$. The proof is complete.
\end{proof}

\vspace{10mm}

\bibliographystyle{apsrev}

\end{document}